\UseRawInputEncoding
\documentclass[11pt, a4paper]{amsart}
\usepackage{amssymb}
\usepackage{amsopn}
\usepackage{amsmath}
\usepackage[colorlinks, citecolor=red]{hyperref}

\usepackage{multirow}
\usepackage[all]{xy}
\usepackage{amsfonts,amsthm,mathrsfs}

\numberwithin{equation}{section}
\newtheorem{theorem}[equation]{Theorem}
\newtheorem*{theorem*}{Theorem}
\newtheorem*{Hodge Conjecture}{Hodge Conjecture}
\newtheorem*{Hodge Conjecture (reduced)}{Hodge Conjecture (reduced)}
\newtheorem*{Lefschetz $(1,1)$-theorem}{Lefschetz $(1,1)$-theorem}
\newtheorem*{Jacobi Inversion Theorem}{Jacobi Inversion Theorem}
\newtheorem*{Claim}{Claim}
\newtheorem*{Base Change Theorem}{Base Change Theorem}
\newtheorem{lemma}[equation]{Lemma}
\newtheorem{corollary}[equation]{Corollary}

\theoremstyle{definition}

\theoremstyle{remark}
\newtheorem*{remark}{Remark}

\theoremstyle{plain}

\newcommand{\opname}[1]{\operatorname{\mathsf{#1}}}

\newcommand{\ie}{{\rm i.e.}}

\newcommand{\Pic}{\opname{Pic}}
\newcommand{\C}{\mathbb{C}}

\newcommand{\Z}{\mathbb{Z}}
\renewcommand{\L}{\mathbb{L}}
\renewcommand{\P}{\mathbb{P}}

\newcommand{\pr}{\mathit{pr}}

\newcommand{\Hom}{\operatorname{Hom}}

\DeclareMathOperator{\im}{im}

\renewcommand{\Im}{\operatorname{Im}}

\DeclareMathOperator{\GL}{GL}

\newcommand{\sg}{\opname{sg}}
\newcommand{\sm}{\opname{sm}}
\newcommand{\F}{\mathbb{F}}

\renewcommand{\O}{\mathbb{O}}
\newcommand{\Mod}{\opname{Mod}}

\renewcommand{\tilde}[1]{\widetilde{#1}}

\newcommand{\M}{\mathbb{F}}
\renewcommand{\a}{\opname{a}}

\renewcommand{\pr}{\opname{pr}}

\renewcommand{\Pic}{\opname{Pic}}

\newcommand{\van}{\opname{van}}
\newcommand{\p}{\opname{p}}

\begin{document}

%\date{September 2018, last modified on \today}

\title[Tube classes over elementary vanishing cycles]{Tube classes over elementary vanishing cycles}
\author{Erjuan Fu}
%\address{University of Utah\\
%    Department of Mathematics\\
%   155 South 1400 East, JWB 328 \\
%   Salt Lake City, UT 84112\\
%  Phone: 801-585-5469}
\address{Yau Mathematical Sciences Center,
Tsinghua University, Beijing, 100084, China}
  %\curraddr{Yau Mathematical Sciences Center, Tsinghua University, Beijing, 100084, China}
 \email{ejfu@tsinghua.edu.cn; 473051732@qq.com}

\begin{abstract}
Let $X$ be a closed Riemann surface. When $X$ is embedded into a projective space, the first rational cohomology group can be concretely obtained from the monodromy in the family of its smooth hyperplane sections by  C. Schnell's  tube mapping. We generalize  this result to the first integral homology group by
relating the tube mapping with the topological Abel--Jacobi mapping. By making use of the mapping class group action, we prove that all tube classes constructed from the elementary vanishing cycles form a cofinite subgroup of  the first integral homology group of $X$.
\end{abstract}

\hypersetup{%colorlinks=true,
linkcolor=blue}
\maketitle
%\tableofcontents

\section{Introduction}\label{s:introduction}
Let $X$ be a closed Riemann surface of genus $g \geq 1$ and $L$ be a line bundle of degree $d\geq 2g+1$ over $X$
which induces a closed embedding $i_L: X \hookrightarrow |L|^\vee:=\P^N$. %\footnote{In this paper, we denote by $(-)^\vee:=\Hom_{\C}(-, \C)$.}%, where $R$ can be $\Z, \Q, \R$ and $\C$.}.
%Let $X$ be a nondegenerate smooth complex project curve in a projective space $\P^N$.
Denote by $\O^{\sm}$ %be the set of all hyperplanes $H$ such that $X\cap H$ is smooth.
the Zariski-open subset of $|L|$, which parametrizes all smooth hyperplane sections of $X$.
Fix $t_0\in \O^{\sm}$ and denote by $$X_{t_0}:=\{x\in X\, \vert \, t_0(x)=0\}=\{P_0, P_1, \cdots, P_{d-1}\}.$$ We have the vanishing homology group $$H_0^{\van}(X_{t_0}, \Z):=\ker\left\{H_0(X_{t_0},\Z) \xrightarrow{j_*} H_0(X, \Z)\right\}=\bigoplus_{i=1}^{d-1} \Z(P_i-P_0),$$ where $j: X_{t_0}\hookrightarrow X$.
Note that any two smooth hyperplane sections are diffeomorphic to each other, it is possible to transport homology classes among nearby ones, which gives rise to an action of the fundamental group $\pi_1(\O^{\sm}, t_0)$ on $H_0^{\van}(X_{t_0}, \Z)$, \ie,
we have the following monodromy representation
\begin{equation*}
\rho_{\van} : \pi_1(\O^{\sm}, t_0) \to \opname{Aut}(H_0^{\van}(X_{t_0}, \Z))\subseteq \opname{Aut}(X_{t_0})=S_d,
\end{equation*}
where $S_d$ denotes the group of permutations on $d$ objects.
With this monodromy action, we have the tube mapping \cite{Tube}
\begin{equation*}
\tau: \bigoplus\limits_{\gamma \in \pi_1(\O^{\sm}, t_0)} \left\{D \in H_0^{\van}(X_{t_0}, \Z)\,\big|\, \gamma D =D \right\} \to H_1(X, \Z).
\end{equation*}
C. Schnell \cite[Theorem 1]{Tube} proved that this tube mapping $\tau$ is surjective over \emph{rational} coefficient homology if $H_0^{\van}(X_{t_0}, \Z)\neq 0$. In this paper, our goal is to generalize this result by considering only elementary vanishing cycles and relating the tube mapping with the topological Abel--Jacobi mapping. By an elementary vanishing cycle, we mean that it specializes an indivisible vanishing cycle which vanishes as approaching to one singular fiber in a family. For instance, $P_i-P_j$ is an elementary vanishing cycle for any $i\neq j$.
Denote by $\a_X^{\opname{top}}$ the topological Abel--Jacobi mapping. The induced topological Abel--Jacobi homomorphism on fundamental group %$\opname{a}_{X*}^{\opname{top}}$
\begin{equation*}\label{top-map}
  \opname{a}_{X*}^{\opname{top}}: \bigoplus\limits_{D \in H_0^{\van}(X_{t_0}, \Z)}\left\{\gamma\in\pi_1(\O^{\sm}, t_0)\,\vert\, \gamma D=D \right\}\to  H_1(X, \Z)
  \end{equation*}
coincides with the tube map $\tau$ up to a sign \cite[Section 3]{Fu}.

Note that  monodromy representations can be computed using a Lefschetz pencil %\cite[Section 2.1]{Voisin2}
by Zariski's Theorem \cite[Theorem 3.22]{Voisin2}, instead of considering all hyperplane sections, we only need to pick a Lefschetz pencil of hyperplane sections.
Let $\L$ be a general projecitve line through $t_0$  in $\O$, %\footnote{In this dissertation, we use $X^\vee$ to denote the projective dual variety of $X$  (\cf Section \ref{ProjDual}).}
then $\{X_t\}_{t\in \L}$ forms a Lefschetz pencil \cite[Corollary 2.10]{Voisin2}, where $X_t=\{x\in X\,|\, t(x)=0\}$. Denote by $\L^{\sm}:=\L\cap \O^{\sm}$ which parametrizes all smooth hyperplane sections of $X$ over $\L$. Our main theorem is stated as follows.

\begin{theorem}\label{thm-main}
Let $X$ be a nondegenerate smooth complex projective curve in $\P^N$ and $\L$ be a Lefschetz pencil of hyperplane sections in $\P^N$. Fix $t_0\in \L^{\sm}$, denote by $X_{t_0}:=\{P_0, \cdots, P_{d-1}\}$ the smooth hyperplane section over $t_0$.
Let $G_{ij}$ be the stabilizer of $P_i-P_j$ in $\pi_1(\L^{\sm}, t_0)$. The image of the tube mapping
$$\tau: \bigoplus\limits_{i=1}^{d-1}H_1(G_{0i}, \Z)%=\bigoplus\limits_{g\in  G^t } V_t^g
\to H_{1}(X,\Z)$$
is $m_dH_1(X,\Z)$, where $m_d$ is some nonzero integer only depending on the degree $d$.
\end{theorem}

Compare to the Strong Tube Theorem in \cite{Fu}, our result works for any complex projective curve which is not necessary a hyperplane section of a  surface.  The key point in our proof is the mapping class group action instead of the monodromy action in \cite{Fu}.

%This generalizes C. Schnell¡¯s result \cite[Theorem 1]{Tube} over rational coefficient to integral coefficient for the curve case.

In the future, we will explore the relation between $m_d$ and $d$ to see if we can show $\mathrm{Im}\tau=H_1(X,\Z)$ by constructing two cycles with intersection number having no square factors.

%The rest of this paper is organized as follows. In Section \ref{sec-top}, we give a concrete expression for the topological Abel--Jacobi mapping and then restate our main theorem.

\subsection*{Acknowledgments}
This work is part of  my Ph.D. thesis.
I am deeply grateful to my advisor Herb Clemens for introducing me this beautiful geometric construction and his invaluable intuition which have provided ideas for many proofs.
 I also would like to express my truly gratitude to my advisor Yuan-Pin Lee for his support and many useful advice. % in guiding my research. %Without their constant help, this work would not have been done.

%%%%%%%%%%%%%%%%%%%%%%%%%%%%%%%%%%%%%%%%%%%%%%%%%%%%%%%
%%% -*-LaTeX-*-
\section{The tube mapping and the topological Abel--Jacobi mapping} \label{sec-top}
 In this section, we will give a concrete expression for the topological Abel--Jacobi mapping and then restate our main theorem.
Let $\Omega$ be the bundle of holomorphic $1$-form on $X$. The Jacobian of $X$  is defined by
$$J(X)=\frac{H^0(X, \Omega)^\vee}{H_1(X, \Z)}\footnote{In this paper, we denote by $(-)^\vee:=\Hom_{\C}(-, \C)$.}%, where $R$ can be $\Z, \Q, \R$ and $\C$.}
$$
which is a complex torus and isomorphic to $\C^g/\Z^{2g}$, and
$$H_1(X, \Z) \hookrightarrow H^0(X, \Omega)^\vee$$
is given by integrating forms over chains.  %If the genus of $X$ is $g$, then $J(X)\cong\C^g/\Z^{2g} $.
%Let $\{\omega_1, \cdots, \omega_g\}$ be a basis of $H^0(X, \Omega)$.
 Fix a base point $p_0\in X$,  the Abel--Jacobi mapping is defined as
 \begin{equation*}
\begin{aligned}
\a_X : X &\to J(X) \\
p  &\mapsto \int_{p_0}^p.
\end{aligned}
\end{equation*}
For any $t\in \O^{\sm}$, denote by $$X_t=\{x\in X\, \vert \, t(x)=0\}:=\{P_1^t, P_2^t, \cdots, P_d^t\}.$$ The vanishing homology group can be described as $$H_0^{\van}(X_t, \Z):=\ker\left\{H_0(X_t,\Z) \xrightarrow{j_*^t} H_0(X, \Z)\right\}=\bigoplus_{i=1}^{d-1} \Z(P_i^t-P_0^t),$$ where $j^t: X_t \hookrightarrow X$.  The Abel--Jacobi mapping can be extended to  the vanishing homology group as
 \begin{align*}
 \opname{a}_{X,t}^{\opname{top}}: H_{0}^{\van}(X_{t}, \Z) &\to J(X),\\
\sum\limits_{i=1}^{d-1} n_i(P_i^t-P_0^t) &\mapsto \sum\limits_{i=1}^{d-1} n_i \int_{P_0^t}^{P_i^t},
 \end{align*}
 which is the topological Abel--Jacobi mapping \cite{Fu, TopAJ}.

 Furthermore, in order to combine the topological Abel--Jacobi map $ \opname{a}_{X,t}^{\opname{top}}$ for  all smooth hyperplane sections together, we consider the incidence variety
  $$X_\L=\{ (x,t)\in X\times \L: x\in X_t\}\cong\opname{Bl}_BX\footnote{$\opname{Bl}_BX$ denotes the blow-up of $X$ along $B$.} =X,$$
 where the second equality is due to  $B=\emptyset$ because the base locus $B$ of a Lefschetz pencil is of codimension 2.
Therefore, we have a projection map $\pi: X \to \L$ which is a branched covering map of degree $d$ with   branch locus $\L^{\sg}=\L\setminus \L^{\sm}$.

Now we consider the local system $\opname{R}^{0}_{\van}\pi_{*}^{\sm}\Z$ over $\L^{\sm}$ which are
fiberwise given by $H^{0}_{\van}(X_{t}, \Z)$ and denote the total space by $ \big|\opname{R}^{0}_{\van}\pi_{*}^{\sm}\Z\big| $,  % which is an infinite-sheeted covering space of $\O_t^{\sm}$,
 then the topological Abel--Jacobi mapping over the local system $\opname{R}^{0}_{\van}\pi_{*}^{\sm}\Z$ becomes
 $$ \opname{a}_X^{\opname{top}}: \big|\opname{R}^{0}_{\van}\pi_{*}^{\sm}\Z\big| \to J(X).$$

 Pick   $\alpha\in H_{0}^{\van}(X_{0}, \Z)$, we consider the induced topological Abel--Jacobi homomorphism on fundamental groups
   $$ \opname{a}_{X*}^{\opname{top}}: \pi_1\left(\big|\opname{R}^{0}_{\van}\pi_{*}^{\sm}\Z\big|, (t_0, \alpha)\right) \to \pi_1(J(X))\cong  H_{1}(X, \Z),$$
where $\pi_1\left(\big|\opname{R}^{0}_{\van}\pi_{*}^{\sm}\Z\big|, (t_0, \alpha)\right)$ is exactly the stabilizer of $\alpha$ in $\pi_1(\L^{\sm}, t_0)$.

We now give a concrete description for $\opname{a}_{X*}^{\opname{top}}$ to relate with the tube mapping $\tau$.
  Since every hyperplane section $X_t$ over $\L$ has at most one ordinary double point as singularity, the ramification index of each branched point is at most 1.   By the Riemann--Hurwitz formula, we have
$$2g-2=d(-2)+\deg \L^{\sg},$$
 then $\L^{\sg}$ is of even degree in $\L$, %$\Sigma_t$ is of even degree,
%$$\deg (\mathscr{O}_{S_t}\otimes K_S) = -2 d +\deg \Sigma_t,$$
 thus we can make a double cover $\sigma: \tilde\L\to \L$ branched along $\L^{\sg}$. %Denote by
 %By Hurwitz formula, $2g(\widetilde\L_t)-2=2(-2)+\deg \Sigma_t$, then $g(\widetilde\L_t)=\deg\Sigma_t/2-1=d+g-2$.
  Denote  by $\tilde X$ the normalization of the fiber product $X \times_{\L} \widetilde\L$.
 Since the monodromy action $\rho_{\van}$ is of order at most $d!$, there is a finite unbranched covering $\rho: \F\to \tilde\L$ such that  the pull back of  $\tilde X$ over $\F$ is trivial, which implies that all vanishing cycles are stable under the monodromy action $\rho_{\van}$. Therefore, $\F$ is the deformation space for any elementary vanishing cycle $P_i-P_j$, then $\pi_1(\F, s_0)=G_{ij}$  the stabilizer of $P_i-P_j$ for all $i\neq j$, where $s_0\in \M$ such that $\sigma\circ\rho(s_0)=t_0 \in \L^{\sm}$.
We have the following commutative diagram
 %Note that $\tau: \M \to \L$ is a finite covering map branched at $\Sigma$ such that the pullback of $X$ over $\M$ is $d$ copies of $\M$.
\begin{equation*}
\xymatrix{{\M\times X_0}\ar[r]\ar[d]^{\tilde\pi_d}
&\tilde{X}\ar[r]\ar[d]& X\ar[d]^\pi \ar@{^{(}->}[r]^{\opname{a}_X}& J(X)\\
\M\ar[r]^\rho \ar@{.>}[urr]|->>>>>{\phi_i}\ar@{.>}[urrr]|->>>>{\varphi_i}& \tilde{\L}\ar[r]^\sigma& \L}.
\end{equation*}
%\begin{equation*}
%\xymatrix{{\M\times X_0}\ar[r]\ar[d]^{\tilde\pi_d}& X\ar[d]^<<\pi \ar[r]^{\opname{a}_X}& J(X)\\
%\M\ar[r]^\rho \ar@{.>}[ur]|-{\phi_i}\ar@{.>}[urr]|-{\varphi_i}& \L}.
%\end{equation*}
 By the lifting criterion, there are exactly $d$ covering maps $\phi_i: \M \to X$ with $\phi_i(s_0)=P_i$, $i=0,1, \cdots, d-1$. %And $\deg \phi_i=2\cdot d!/d=2 \cdot (d-1)!$.
 Since $\pi_1(\M, s_0)$ trivially acts on $H_0(X, \Z)$, \ie, $\gamma (P_i)=P_i$ for all $\gamma \in \pi_1(\M, s_0)$, we can form the tube class $\tau_\gamma (P_i) \in H_1(X, \Z)$ and $\phi_i$ induces a group homomorphism
\begin{equation*}
\begin{aligned}
\phi_{i*} : H_1(\M, \Z) &\to H_1(X, \Z)\\
\gamma &\mapsto \phi_i(\gamma)=\tau_\gamma(P_i).
\end{aligned}
\end{equation*}
Note that $H_1(\M, \Z)$ is the abelianization of $\pi_1(\F, s_0)=G_{ij}$, then $H_1(\M, \Z)\cong H_1(G_{ij}, \Z)$ for all $i\neq j$.
So $\phi_{ij*}=\phi_{i*}-\phi_{j*}$ is exactly the topological Abel--Jacobi homomorphism $\a_{X*}^{\opname{top}}$ with respect to the vanishing cycle $P_i-P_j$ and it is the same as the tube mapping $\tau$ by  definition.
%$\tau=\sum_{i=1}^{d-1}\phi_{ij*}$.
Denote by $\varphi_i=\opname{a}_X\circ \phi_i$, then $\varphi_{ij}=\opname{a}_X\circ\phi_{ij}$.  We can identify $\varphi_{ij*}$ with $\phi_{ij*}$ because $\opname{a}_X$ is an embedding.
%We shall prove that the tube mapping $\tau$ over the integral coefficient has finite cokernal.
So Theorem \ref{thm-main} is equivalent to the following theorem.
\begin{theorem}[=Theorem \ref{thm-main}]\label{thm-main-1}
 The image of the tube mapping
$$\tau=\sum\limits_{i=1}^{d-1}\varphi_{0i*} : \bigoplus\limits_{i=1}^{d-1}H_1(\F, \Z) %=\bigoplus\limits_{g\in  G^t } V_t^g
\to H_{1}(X,\Z)$$
is $m_dH_1(X,\Z)$, where $m_d$ is some nonzero integer only depending on the degree $d$.
\end{theorem}

\section{Proof of the main theorm}
% In what follows, we denote by $V_{(X, L,\L)}$ the image of $\tau$.
Consider the mapping class group of $X$, which is the group of isotopy (or homotopy) classes of orientation-preserving diffeomorphisms of $X$, denoted by
$$\Mod(X)=\opname{Diff}^+(X)/\sim,$$
where $\opname{Diff}^+(X)$ is the group of orientation-preserving diffeomorphisms of $X$ and $\sim$ can be taken to be either smooth homotopy or smooth isotopy.
Note that any element $\phi \in \opname{Diff}^+(X)$ induces an automorphism $\phi_* : H_1(X, \Z) \to H_1(X, \Z)$ which preserves the intersection form $<\cdot, \cdot>$. As homotopic diffeomorphisms $\phi \sim \psi$ induce the same map $\phi_*=\psi_*$, there is a representation
$$ \rho : \Mod(X) \to \opname{Aut}(H_1(X, \Z), <\cdot, \cdot>)\cong \opname{Sp}_g(\Z),$$
where $\opname{Sp}_g(\Z)$ denotes the symplectic group of degree $g$ over $\Z$.
We have the following conclusion about this mapping class group action, which will be proved in Section \ref{s:stable-mcg}.
\begin{lemma}\label{thm-stable-mcg}
The invariant $\Z$-submodules of $H_1(X,\Z)$ under the representation
$\rho$
are $$mH_1(X,\Z), \qquad m\in \Z.$$
\end{lemma}
Note that $\tau$ is not zero by \cite[Lemma 4.1]{Fu},
 our main theorem (\ie, Theorem \ref{thm-main-1}) is directly obtained from this lemma if we could show $\Im \tau:=V_{(X, L,\L)}$ is stable under the mapping class group action $\rho$.  That is to say, we only need to show the following lemma.  %\ie, $f_*(V_{(L,\L)})=V_{(L,\L)}$ for all $f\in \opname{Diff}^+(X)$.
\begin{lemma}\label{thm-prim-diff}
The $\Z$-submodule $V_{(X, L,\L)}$ is invariant under any orientation-preserving diffeomorphism, \ie, $f_*(V_{(X, L,\L)})=V_{(X, L,\L)}$ for all $f\in \opname{Diff}^+(X)$.
\end{lemma}

%In the future, we will explore the relation between $m_d$ and $d$ to see whether we can show $\Im \tau= H_1(X,\Z)$ by constructing two cycles with intersection number having no square factors.

%%%%%%%%%%%%%%%%%%%%%%%%%%%%%%%%%%%%%%%%%%%%%%%%%%%

%\subsection{Proof of Lemma \ref{thm-prim-diff}}
%Stability under orientation-preserving diffeomorphisms}\label{s:stable-mcg}
\begin{proof}
First of all, since $V_{(X, L, \L)}$ is independent of the choice of the Lefschetz pencil $\L$ by \cite[Lemma 4.1]{Fu},  we can denote $V_{(X, L, \L)}$ by $V_{(X,L)}$.
Furthermore, we claim that
%\item $V_{(X, L, \L)}$ is independent of the choice of the Lefschetz pencil $\L$, thus we can denote by $V_{(X,L)}$, which is true by Theorem \ref{thm-top-2};  %in fact, we proved $\im \varphi_{ij*}^{\L}$ is independent of the choice of $\L$;
\begin{Claim}
$V_{(X, L)}$ is independent of the choice of the degree $d$ line bundle $L$.
%(In fact, we prove $\im \varphi_{ij*}^L$ is independent of the choice of $L$ in Section \ref{ss: IndLbd}).}
\end{Claim}
 \noindent With this claim, $V_{(X, L)}$ can be denoted by $V_{(X, d)}$. % and is stable under orientation-preserving diffeomorphisms.
  %Indeed,  if $V_{(X, L)}$ does not depend on the choice of the Lefschetz pencil $\L$ and the line bundle $L$ of degree $d$, we can denote by $V_{(X,d)}:=V_{(X, L)}$.
  Since $d$ is fixed,
   we denote by $V_X:=V_{(X, d)}$ in what follows.

  Let $f : X \to Y$ be an orientation-preserving diffeomorphism between two closed Riemann surfaces of genus $g$. We only need to show $f_*(V_X)=V_Y$. % and $d\geq 2g+1$ is a fixed positive integer,
%Note that $f_*L\in \Pic^d(Y)$ for any $L\in \Pic^d(X)$. Therefore, we need to show $V_{(Y, f_*L)}=V_{(X, L)}$.

Pick any $L\in \Pic^d(Y)$, we have $f^*L\in \Pic^d(X)$.
Denote by $i_L : Y \hookrightarrow \P^N$  the closed embedding induced by $L$,  then $f^*L$ induces a closed embedding $i_{f^*L}: X \hookrightarrow \P^N$.
%By the construction of
Since $V_X$ and  $V_Y$ contain the tube classes of all elementary vanishing cycles, we have $i_{f^*L*}: V_X \xrightarrow{\sim} V_{i_{f^*L}(X)}$ and $i_{L*}: V_Y \xrightarrow{\sim} V_{i_L(Y)}$ such that the following diagram commutes
$$\xymatrix{X \ar[r]^f\ar@{^{(}->}[rd]_{i_{f^*L}} & Y \ar@{^{(}->}[d]^{i_L} \\ & \P^N}.$$
We obtain $i_L(Y)=i_{f^*L}(X)$. %By the construction of $V_X$ and $V_Y$, we have
Since $V_{i_L(Y)}=V_{i_{f^*L}(X)}$,  we have
$$i_{f^*L*}(V_X)=V_{i_{f^*L}(X)}=V_{i_L(Y)}=i_{L*}(V_Y).$$
Note that $i_{f^*L}=i_L\circ f$, then $i_{f^*L*}=(i_L\circ f)_*=i_{L*}\circ f_*$, which gives us $i_{L*}\circ f_*(V_X)=i_{L*}(V_Y)$. Since $i_{L}$ is an embedding and $i_{L*}|_{V_Y}$ is  injective,  we have $f_*(V_X)=V_Y$.

\medskip

%%%%%%%%%%%%%%%%%%%%%%%%%%%%%%%%%%%%%%%%%%%%%

%%%%%%%%%%%%%%%%%%%%%%%%%%%%%%%%%%%%%%%%%%%%%%%%%%%

\noindent{\bf Proof of Claim.}
%{Independent of the choice of line bundle of fixed degree $d\geq 2g+1$}\label{ss: IndLbd}
Note that $$V_{(X,L)}=\sum\limits_{i=1}^{d-1}\Im\varphi_{0i*}^L,$$
we only need to prove $\Im \varphi_{ij*}^L$ is independent of the choice of $L$.
%Note that the image of $\varphi_{ij*}^{\mathbb{L}}$ is independent of the choice of the Lefschetz pencil $\L$ from \cite[Lemma 4.1]{Fu}, we can denote $\im \varphi_{ij*}^{\mathbb{L}} = \im \varphi_{ij*}^L$ and define
Let
\begin{equation*}
\begin{aligned}
\eta_{ij}:  \opname{Pic}^{d}(X) &\to  \opname{M}_{2g}(\Z)/{\GL_{2g}(\Z)}\\
L &\mapsto\eta_{ij}({L})=\Im \varphi_{ij*}^{{L}}.
\end{aligned}
\end{equation*}
Clearly, $\Im \varphi_{ij*}^{{L}}$ does not depend on the choice of the line bundle $L$ of degree $d$ if and only if $\eta_{ij}$ is a constant map.  Let $\opname{Pic}^d(X)$ be  the set of degree $d$ line bundles on $X$, which is a compact complex torus due to the following isomorphisms
 $$\opname{Pic}^d(X)  \xrightarrow[\sim]{-\otimes \mathscr{O}(dP)}  \opname{Pic}^0(X) \xrightarrow[\sim]{\a_X} J(X),$$%by the following isomorphism
%$$\opname{Pic}^0(X) \xrightarrow{-\otimes \mathscr{O}(dP)} \opname{Pic}^d(X)\hskip 20pt \text{and} \hskip 20pt\opname{Pic}^0(X) \xleftarrow{-\otimes \mathscr{O}(-dP)} \opname{Pic}^d(X),$
where $P\in X$. Clearly, $\opname{Pic}^d(X)$  is path connected and $\opname{M}_{2g}(\Z)/{\GL_{2g}(\Z)}$ is discrete,  % and thus connected,
so $\eta_{ij}$ is constant if and only if it is continuous.  %Therefore, we only need to show $\eta_{ij}$ is continuous.

For any $L\in \opname{Pic}^d(X)$, since $d\geq 2g+1$,  there is a closed embedding $i_L : X \to \P^N$ %with $N=\ell(L)-1=d-g$
induced by $L$. Consider the following incidence variety
$$\mathfrak{P} : = \left\{ (x, L) \in \P^N\times \Pic^d(X) \, | \, x \in i_L(X)\right\}
%=\bigcup\limits_{L \in \opname{Pic}^d(X) } i_L(X)
\xrightarrow{\opname{q}} \opname{Pic}^d(X),$$
 which is a fiber bundle over $\opname{Pic}^d(X)$ with fiber $\opname{q}^{-1}(L)=i_L(X) $.  %, where $ \pi:\mathfrak{P} \rightarrow \opname{Pic}^d(X)$.
 If we can show $\mathfrak{P}$ is locally topological trivial over $\opname{Pic}^d(X)$, \ie, for any $L\in \opname{Pic}^d(X)$, there is an open neighborhood $U_L$ of $L$ such that $\mathfrak{P}|_{U_L} \cong U_L\times i_L(X)$,  then
  $i_{L'}(X)$ continuously goes to $i_L(X)$ when $L'$ continuously goes to $L$ in $U_L$.  Hence,
$\eta_{ij}({L}')=\im \varphi_{ij*}^{{L}'}$ continuously goes to $\eta_{ij}({L})=\im \varphi_{ij*}^{{L}} $, which means $\eta_{ij}$ is continuous at $L$. Since $L$ is arbitrary,  $\eta_{ij}$ is continuous.
 % Consequently, we only need to show  $\mathfrak{P}$ is locally topological trivial over $\opname{Pic}^d(X)$.

 Consider the \emph{Poincar\'e line bundle} $\mathscr{L}$ of degree  $d$ on $X$, which means a line bundle $\mathscr{L}$ on $X \times \opname{Pic}^d(X)$ such that  $e_*(\mathscr{L}|_{X\times [D]} )\cong \mathcal{O}(D)$ for each point  $[D]\in \opname{Pic}^d(X)$, where $ e: X\times [D] \xrightarrow{\sim} X$.
Denote by $\p : X \times \opname{Pic}^d(X) \to  \opname{Pic}^d(X)$.  As we know, $\p_*\mathscr{L}$ is not necessarily a vector bundle on $\opname{Pic}^d(X)$, but we will see it is a vector bundle for $d\geq 2g-1$.

\begin{lemma}\label{lem-vb}
If $d\geq 2g-1$, then $\p_*\mathscr{L}$ is a vector bundle on $\opname{Pic}^d(X)$ with fibers $(\p_*\mathscr{L} )_{[D]} \cong H^0(X,\mathcal{O}(D))$.
\end{lemma}

\begin{proof}
Denote by $K$ the canonical divisor on $X$, then $\deg K=2g-2$.
Since $|D|\cong (\Gamma(X, \mathcal{O}(D))-0)/{\C^*}$,  we have $\dim |D| =\ell(D)-1$\footnote{Denote by $\ell(D)=\dim H^0(X, \mathcal{O}(D))$.}.

If $\deg D \geq 2g-1$, then $\deg(K-D)<0$. There is not a positive divisor linearly equivalent to $K-D$, thus $\Gamma(X, \mathcal{O}(K-D))=0$ and then $\ell(K-D)=0$.
By the Riemann--Roch Theorem, $\ell(D)-\ell(K-D)=d+1-g$, then $\ell(D)=d+1-g$, which does not depend on the choice of $[D]$ in $\opname{Pic}^d(X)$, which gives us a constant function on $\opname{Pic}^d(X)$ defined by $[D] \mapsto \ell(D)$.
Now we need to use the Base Change Theorem  \cite[Corollary 2 on p. 50]{Mumford-AV}.
%and the proof can be found in Mumford's Abelian Varieties, Chapter II:
\begin{Base Change Theorem}
Let $f: X \to Y$ be a proper morphism of Noetherian schemes with $Y$ reduced and connected, and $E$ a coherent sheaf on $X$, flat over $Y$. Then for all integers $k\geq 0$, the following conditions are equivalent:
\begin{enumerate}
\item[(1)] $y\mapsto \dim H^k(X_y, E_y)$ is a constant function, where $X_y:=f^{-1}(y)$ and $E_y:=E|_{X_y}$.
\item[(2)] $R^kf_*E$ is a locally free sheaf on $Y$ and the natural map
$R^kf_*E \otimes_{\mathcal{O}_Y}k(y) \to H^k(X_y, E_y)$
is an isomorphism for all $y\in Y$.
\end{enumerate}
\end{Base Change Theorem}

Apply this theorem to $\p : X \times \opname{Pic}^d(X) \to  \opname{Pic}^d(X)$, we have $\p^{-1}([D])=X \times [D]\cong X$ and $\mathscr{L}|_{X\times [D]} \cong e^*\mathcal{O}(D)$. Furthermore,
 if $[D] \mapsto \dim H^0(X, \mathcal{O}(D)):=\ell(D)$ is a constant function on $\opname{Pic}^d(X)$,  then $\p_*\mathscr{L}$ is a vector bundle on $\opname{Pic}^d(X)$ with fibers $(\p_*\mathscr{L} )_{[D]} \cong H^0(X,\mathcal{O}(D))$, which is desired.
%Therefore,  $\pi_*\mathcal{L}$ is a vector bundle on $\opname{Pic}^d(X)$ with fibers $(\pi_*\mathcal{L} )_{[D]} \cong H^0(X,\mathscr{O}(D))$.
\end{proof}

%Note that the natural map $ X^{(d)}  \to \opname{Pic}^d(X)$ has fibers $X^{(d)} _{[D]}=|D|\cong (H^0(X, \mathscr{O}(D))-0)/{\C^*} =\mathbb{P}(\pi_*\mathcal{L} )_{[D]} $, then $X^{(d)}=\mathbb{P}(\pi_*\mathcal{L})$. Where $X^{(d)}$ is the $d$-th symmetric product of $X$.

We now have the following commutative diagram
$$
\xymatrix{\mathscr{L}\ar[d]& \pr_{2*}\mathscr{L}\ar[d]\\
X\times \Pic^d(X) \ar@{_{(}->}[d]\ar[r]^{\pr_2}& \Pic^d(X)\\
\mathfrak{P} \ar@{_{(}->}[r]\ar[ru]^\p& \P^N\times \Pic^d(X)}.
$$
%where $\mathscr{L}$ denotes the Poincar\'e line bundle over $X\times \Pic^d(X)$.     % then  $e_*(\mathscr{L}|_{X\times L})\cong L$ with $e : X\times L \xrightarrow{\sim} X $.
Since $d\geq 2g+1$,  we have that $\opname{pr}_{2*}\mathscr{L}$ is a vector bundle over  $\opname{Pic}^d(X)$ with fibers $(\pr_{2*}\mathscr{L} )_{L} \cong H^0(X, L)$ by Lemma \ref{lem-vb}.  There is a trivialization for $\opname{pr}_{2*}\mathscr{L}$, \ie, for each $L\in \opname{Pic}^d(X)$, there is an open neighborhood $U_L$ of $L$ in $\opname{Pic}^d(X)$ with a biholomorphism $\xi_{U_L} : \pr_2^{-1}(U_L) \xrightarrow{\sim}  \C^{N+1} \times U_L$. Denote by $\sigma_i(L')=\xi_{U_L}^{-1}(L', e_{i})$ for $L'\in U_L$, where $\{ e_0, \cdots, e_N\}$ is the canonical basis of $\C^{N+1}$, then the holomorphic sections $\sigma_0, \cdots, \sigma_N$ of $\pr_{2*}\mathscr{L}$ over $U_L$ form a frame, \ie, $\{\sigma_0(L'), \cdots, \sigma_N(L')\}$ is a basis of $(\pr_{2*}\mathscr{L})_{L'}=H^0(X, L')$. Consequently, $i_{L'}$ with $L'\in U_L$ can be written as
\begin{equation*}
\begin{aligned}
i_{L'} : X &\xrightarrow{\sim} i_{L'}(X) \hookrightarrow \P^N\\
x &\mapsto [\sigma_0(L')(x), \cdots, \sigma_N(L')(x)]=i_{L'}(x),
\end{aligned}
\end{equation*}
which gives us a trivialization of $\mathfrak{P}$ over $\Pic^d(X)$
\begin{equation*}
\begin{aligned}
\zeta_{U_L} : \p^{-1}(U_L) &\to i_L(X)\times U_L\\
\left( [\sigma_0(L')(x), \cdots, \sigma_N(L')(x)], L'\right)&\mapsto \left([\sigma_0(L)(x), \cdots, \sigma_N(L)(x)], L'\right).
\end{aligned}
\end{equation*}
Therefore, $\mathfrak{P}$ is locally topological trivial over $\opname{Pic}^d(X)$.
\end{proof}

\begin{remark}
From the proof, we see that it is necessary to use all elementary vanishing cycles, because the orientation-preserving diffeomorphism may convert one elementary vanishing cycle to the other, \ie, it will map $\im \varphi_{ij}$ to $\im \varphi_{k\ell}$. Therefore,  we cannot use only one elementary vanishing cycle to generate the whole homology group under the mapping class group action.
\end{remark}

%%%%%%%%%%%%%%%%%%%%%%%%%%%%%%%%%%%%%%%%%%%%%%%%%%%%%%%%%%%%%%%%%%%%%%

\subsection{Stable $\Z$-submodules under the mapping class group action}\label{s:stable-mcg}
Let $[c]$ be the homology class corresponding to an oriented simple closed curve $c$.  By a closed curve in a surface $X$, we mean a continuous map $S^1 \to X$.  A closed curve is simple if it is embedded, that is, if the map $S^1 \to X$ is injective.  A closed curve is multiple if the map $S^1 \to X$ factors through the map $S^1 \xrightarrow{\times n} S^1$ for some integer $n>1$.
Denote by $T_c$ the Dehn twist about $c$, where $c$ is the isotopy class of an oriented simple closed curve in $X$, then $T_c\in \Mod(X)$ and $\Mod(X)$ is generated by a finitely many of Dehn twists. We have the following formula for Dehn twists action {\cite[Proposition 6.3]{MCG}}.

\begin{lemma}\label{formula}
Let $a$ and $b$ be isotopy classes of oriented simple closed curves in $X$. For any integer $k\geq 0$, we have
\begin{equation*}
\rho(T_b^k)([a])=[a]+k<[a],[b]>[b]
\end{equation*}
\end{lemma}
By Lemma \ref{formula}, we get
$$\rho(T_{a})=\rho(T_b) \Leftrightarrow [a]=[b].$$
Hence, we can denote by $T_{[b]}:=\rho(T_b)\in \opname{Aut}(H_1(X, \Z), <\cdot, \cdot>)$ in what follows.

An element $g$ of a group $G$ is \textit{primitive} if there does not exist any $h\in G$ so that $g=h^k$ for some integer $k$ with $|k|>1$. Therefore, $\xi \in H_1(X, \Z)$ is primitive if $\xi\neq k \eta$ for any $\eta \in H_1(X, \Z)$ and any integer $k>1$.  The primitive elements in $H_1(X, \Z)$ can be described as follows \cite[Proposition 6.2]{MCG}.
\begin{lemma}\label{primitive}
A nonzero element of $H_1(X, \Z)$ is represented by an oriented simple closed curve if and only if it is primitive.
\end{lemma}

By Lemma \ref{primitive}, we have  a group automorphism $T_\alpha \in \opname{Aut}(H_1(X, \Z), <\cdot, \cdot>)$ for any nonzero primitive element $\alpha \in H_1(X, \Z)$.
Combining Lemma \ref{formula} with Lemma \ref{primitive} gives us the following corollary.

\begin{corollary}\label{newformula}
Let $\alpha$ and $\beta$ be two primitive elements in $H_1(X, \Z)$. For any integer $k\geq 0$, we have
\begin{equation}\label{ourformula}
T_\beta^k(\alpha)=\alpha+k<\alpha,\beta>\beta
\end{equation}
\end{corollary}
Now, we will use (\ref{ourformula}) to prove Lemma \ref{thm-stable-mcg}.
%\begin{theorem}\label{invariant}
%The stable $\Z$-submodules of $H_1(X,\Z)$ under the representation
%$$ \rho : \Mod(X) \to \opname{Aut}(H_1(X, \Z), <\cdot, \cdot>)$$
%are $mH_1(X,\Z)$, $m\in \Z$.
%\end{theorem}
\vskip 5pt

\noindent{\bf \it Proof of Lemma \ref{thm-stable-mcg}.}\,
Let $V$ be a nonzero stable $\Z$-submodule of $H_1(X, \Z)$, then we only need to show $V=mH_1(X, \Z)$ for some nonzero integer $m$.

Pick a symplectic basis $\{\delta_i, \gamma_i\}_{i=1}^g$ of $H_1(X, \Z)$ with $<\delta_i, \gamma_j>=\delta_{ij}$ and $<\delta_i, \delta_j > =<\gamma_i, \gamma_j>=0$ for all $i, j$.  Note that  $<\delta_i, \gamma_i>=1$ , then $\{\delta_i, \gamma_i\}_{i=1}^g$ are primitive. Since $<\delta_i+\gamma_j,\gamma_i>=1$ and $<\delta_k+\delta_l, \gamma_k>=<\gamma_k+\gamma_l, \delta_k>=1$ for $k\neq l$, we get that $\delta_i+\gamma_j, \delta_k+\delta_l, \gamma_k+\gamma_l $ with $k\neq l$ are primitive.  By (\ref{ourformula}), we have the following formulas.
 \begin{equation}\label{base}
 \begin{aligned}
T_{\delta_j}(\delta_i)&=\delta_i  &\qquad T_{\gamma_j}(\gamma_i)&=\gamma_i\\
T_{\delta_j}(\gamma_i)&=\left\{
\begin{aligned}
&\gamma_i  &\text{for} \quad i\neq j \\
&\gamma_i-\delta_i  &\text{for} \quad i= j
\end{aligned} \right. &
T_{\gamma_j}(\delta_i)&=\left\{
\begin{aligned}
&\delta_i & \text{for} \quad i\neq j \\
 &\delta_i+\gamma_i  & \text{for} \quad i= j
 \end{aligned}\right.
 \end{aligned}
 \end{equation}

Pick a nonzero element $\alpha=\sum_{i=1}^g(a_i\delta_i+b_i\gamma_i)\in V$ with $a_i, b_i\in \Z$.   Denote by $d_\alpha=\gcd(a_1,b_1,\cdots, a_g,b_g)\geq 1$.
We separate the proof into the following 4 steps.
\vskip 5pt
\begin{itemize}
\item Step 1. \, Show $a_k H_1(X, \Z) \subseteq V$ for all $k=1, \cdots, g$.
\vskip 5pt
\item Step 2. \,  Show $b_k H_1(X, \Z) \subseteq V$ for all $k=1, \cdots, g$.
\vskip 5pt
\item Step 3. \, Show $d_\alpha H_1(X, \Z) \subseteq V$.
\vskip 5pt
\item Step 4. \,  Conclude $V= mH_1(X,\Z)$ for some positive integer $m$.
\end{itemize}
\vskip 5pt
%In order to simplify the notation, we denote by $T_{b}([a]) : = \rho(T_{b})[a]$.
%By the formula in Proposition \ref{formula}, we have
%\begin{equation}\label{newformula}
% T_{b}([a])=[a]+<a, b> [b]
% \end{equation}
% Pick a symplectic basis $\{[\delta_1], \cdots, [\delta_g], [\gamma_1], \cdots, [\gamma_g]\}$ of $H_1(X,\Z)$ with
%$<\delta_i, \gamma_j>=\delta_i^j$, $<\delta_i, \delta_j>=0=<\gamma_i, \gamma_j>$.

%Since $<\delta_i, \gamma_i>=1$ ,  $\{\delta_i, \gamma_i\}_{i=1}^g$ are primitive,

{\bf Step 1. Show $a_k H_1(X, \Z) \subseteq V$ for all $k=1, \cdots, g$.}

 Recall $\alpha=\sum\limits_{i=1}^g(a_i\delta_i+b_i\gamma_i)\in V$ with $a_i, b_i\in \Z$ and note that $T_{\delta_i}$ and $T_{\gamma_j}$ are group automorphisms of $H_1(X, \Z)$.
 By (\ref{base}), we have
\begin{equation*}
\begin{aligned}
T_{\gamma_k}(\alpha)&=a_k\left(T_{\gamma_k}\left(\delta_k\right)\right)+b_k\left(T_{\gamma_k}(\gamma_k)\right)+\sum\limits_{i\neq k}\left(a_iT_{\gamma_k}(\delta_i)+b_iT_{\gamma_k}(\gamma_i)\right)\\
&=a_k\left(\delta_k+\gamma_k\right)+b_k\gamma_k
+\sum\limits_{i\neq k}\left(a_i\delta_i+b_i\gamma_i\right)\\
&=a_k\gamma_k
+\sum\limits_{i =1}^g\left(a_i\delta_i+b_i\gamma_i\right)=a_k\gamma_k+\alpha \in V.
\end{aligned}
\end{equation*}
Therefore, we have $a_k\gamma_k =T_{\gamma_k}(\alpha) -\alpha \in V$, $k=1, \cdots, g$.
 Since $<\delta_i+\gamma_j,\gamma_i>=1$ and $<\delta_k+\delta_l, \gamma_k>=1$ for $k\neq l$, we get that $\delta_i+\gamma_j, \delta_k+\delta_l$ with $k\neq l$ are primitive. Continue to apply the Dehn twists to each $a_k\gamma_k$, we have
\begin{equation*}
\begin{aligned}
T_{\delta_k}(a_k\gamma_k)&=a_k\gamma_k- a_k\delta_k\in V,\\
T_{\delta_k+\delta_i}\left(a_k\gamma_k\right)
&=a_k\gamma_k - a_k\delta_k-a_k\delta_i\in V  \quad \text{for} \,\,  i\neq k, \\
T_{\delta_k+\gamma_i}\left(a_k\gamma_k\right)
&=a_k\gamma_k - a_k\delta_k-a_k\gamma_i \in V.
\end{aligned}
\end{equation*}
Therefore, we get $a_k \delta_1, \cdots, a_k \delta_g, a_k\gamma_1, \cdots, a_k\gamma_g \in V$ for all $k=1, \cdots, g$ and then $$a_k H_1(X, \Z) \subseteq V\quad \text{for all}\quad k=1, \cdots, g.$$

{\bf Step 2. Show $b_k H_1(X, \Z) \subseteq V$ for all $k=1, \cdots, g$.}

Similarly, redo the above calculation for $\delta_k$, we have
\begin{equation*}
\begin{aligned}
T_{\delta_k}(\alpha)&=a_k\left(T_{\delta_k}\left(\delta_k\right)\right)+b_k\left(T_{\delta_k}(\gamma_k)\right)+\sum\limits_{i\neq k}\left(a_iT_{\delta_k}(\delta_i)+b_iT_{\delta_k}(\gamma_i)\right)\\
&=a_k\delta_k+b_k\left(\gamma_k-\delta_k\right)
+\sum\limits_{i\neq k}\left(a_i\delta_i+b_i\gamma_i\right)\\
&=-b_k\delta_k
+\sum\limits_{i =1}^g\left(a_i\delta_i+b_i\gamma_i\right)=-b_k\delta_k+\alpha \in V.
\end{aligned}
\end{equation*}
Therefore, we have $b_k\delta_k = \alpha - T_{\delta_k}(\alpha) \in V$, $k=1, \cdots, g$.
Since $<\delta_i+\gamma_j,\gamma_i>=1$ and $<\gamma_k+\gamma_l, \delta_k>=1$ for $k\neq l$, we get that $\delta_i+\gamma_j, \gamma_k+\gamma_l $ with $k\neq l$ are primitive.
Continue to apply the Dehn twists to each $b_k\delta_k$, we have
\begin{equation*}
\begin{aligned}
T_{\gamma_k}(b_k\delta_k)&=
%b_kT_{\gamma_k}([\delta_k]) = b_k\left([\delta_k]+ [\gamma\right)=
b_k\delta_k+ b_k\gamma_k\in V,
% \Rightarrow b_k[\gamma_k] =T_{\gamma_k}(b_k[\delta_k])- b_k[\delta_k] \in V
\\
T_{\delta_i+\gamma_k}\left(b_k\delta_k\right)
&=b_k\delta_k + b_k\gamma_k+b_k\delta_i\in V \quad \text{for} \quad i\neq k,\\
T_{\gamma_i+\gamma_k}\left(b_k\delta_k\right)
&=b_k\delta_k + b_k\gamma_k+b_k\gamma_i \in V.
\end{aligned}
\end{equation*}
Therefore, we get $b_k \delta_1, \cdots, b_k \delta_g, b_k\gamma_1, \cdots, b_k\gamma_g\in V$ for all $k=1, \cdots, g$ and then $$b_k H_1(X, \Z) \subseteq V\quad \text{for all} \quad k=1, \cdots, g.$$

{\bf Step 3. Show $d_\alpha H_1(X, \Z) \subseteq V$.}

Since $d_\alpha=\gcd(a_1, b_1, \cdots, a_g, b_g)$, there are integers $u_1, v_1, \cdots, u_g, v_g$ such that $d_\alpha=\sum_{i=1}^g(u_ia_i+v_ib_i)$.  For any $k\in \{1, \cdots, g\}$, we have
\begin{equation*}
\begin{aligned}
d_\alpha\delta_k=\sum\limits_{i=1}^g\left(u_ia_i\delta_k+v_ib_i\delta_k\right) \in V,\\
d_\alpha\gamma_k=\sum\limits_{i=1}^g\left(u_ia_i\gamma_k+v_ib_i\gamma_k\right) \in V.
\end{aligned}
\end{equation*}
Therefore, $d_\alpha H_1(X, \Z) \subseteq V$.

{\bf Step 4. Conclude $V= mH_1(X,\Z)$ for some positive integer $m$.}

 If $d_\alpha H_1(X, \Z) = V$, then we are done.
 Otherwise, there exists
 $$\beta=\sum\limits_{i=1}^g(e_i[\delta_i]+f_i[\gamma_i]) \in V\setminus d_\alpha H_1(X, \Z), \qquad e_i, f_i \in \Z.$$
 Denote by $d_\beta=\gcd(e_1, f_1, \cdots, e_g, f_g)$, then $d_\alpha\nmid d_\beta$.  By above analysis, we get $$d_\beta H_1(X, \Z) \subseteq V.$$
 Clearly, $\gamma:=d_\alpha\delta_1+d_\beta \gamma_1 \in V$ and $d_\gamma =\gcd (d_\alpha, d_\delta)<d_\alpha$. Again, by the above analysis, we obtain $d_\gamma H_1(X, \Z) \subseteq V$.
 Therefore, we have
 $$d_\alpha H_1(X,\Z) \subsetneq d_\gamma H_1(X, \Z) \subseteq V\subseteq H_1(X, \Z) \qquad \text{with}\qquad d_\alpha > d_\gamma \geq 1, \quad d_\gamma | d_\alpha.$$
 Replace $\alpha$ by $\gamma$,  do the above discussion in this step again. Since each process will give us a smaller positive integer, by finitely many processes, we will get $V= mH_1(X,\Z)$ for some positive integer $m$, which is desired. \qquad\qquad  \qquad\qquad \qquad\qquad \qquad\qquad \qquad\quad
$\square$

%%%%%%%%%%%%%%%%%%%%%%%%%%%%%%%%%%%%%%%%%%%%%%%%%%%

\end{document}